\documentclass{article}
\usepackage{amsmath}
\usepackage{amsfonts}
\usepackage{amssymb}

\newtheorem{theorem}{Theorem}

\newtheorem{proposition}[theorem]{Proposition}

\newenvironment{proof}[1][Proof]{\noindent\textbf{#1.} }{\ \rule{0.5em}{0.5em}}

\begin{document}

\title{On a family of Levy processes without support in $\mathcal{S}^{\prime
}$}
\author{R. Vilela Mendes\thanks{%
rvilela.mendes@gmail.com; rvmendes@fc.ul.pt;
https://label2.tecnico.ulisboa.pt/vilela/} \\
CMAFcIO, Faculdade de Ci\^{e}ncias, Univ. Lisboa}
\date{ }
\maketitle

\begin{abstract}
The distributional support of the sample paths of L\'{e}vy processes is an
important issue for the construction of sparse statistical models, theories
of integration in infinite dimensions and the existence of generalized
solutions of stochastic partial differential equations driven by L\'{e}vy
white noise. Here one considers a family $K_{\alpha }$ $\left( 0<\alpha
<2\right) $ of L\'{e}vy processes which have no support in $\mathcal{S}%
^{\prime }$. For $1<\alpha <2$ they are supported in $\mathcal{K}^{\prime }$%
, the space of distributions of exponential type and for $0<\alpha \leq 1$
on similar spaces of power exponential type.
\end{abstract}

\section{Introduction}

Characterization of the support of paths of L\'{e}vy processes is an
important issue both for the construction of sparse statistical models \cite%
{Fageot}, for theories of integration in infinite dimensions \cite{Nunno}
and for the existence of generalized solutions of stochastic partial
differential equations driven by L\'{e}vy white noise \cite{Fageot-2}.

All c\`{a}dl\`{a}g L\'{e}vy processes, being locally Lebesgue integrable,
have support in $\mathcal{D}^{\prime }$, the space of distributions (dual to
the space $\mathcal{D}$ of infinitely differential functions with compact
support). However it turns out that the paths of most L\'{e}vy processes
have support on a smaller space $\mathcal{S}^{\prime }$, the space of
tempered distributions, dual to the space $\mathcal{S}$ of rapid decrease
functions, topologized by the family of norms%
\begin{equation}
\left\Vert \varphi \right\Vert _{p,r}=\sup_{x\in \mathbb{R}}\left\vert
x^{p}\varphi ^{(r)}\left( x\right) \right\vert ,\;p,r\in \mathbb{N}_{0}
\label{1.1}
\end{equation}%
Necessary and sufficient conditions have been obtained for the support of a L%
\'{e}vy process to be in $\mathcal{S}^{\prime }$ \cite{Lee-Shih} \cite%
{Dalang-Humeau}. In particular \cite{Dalang-Humeau}, a L\'{e}vy process $%
X_{t}$ has support in $\mathcal{S}^{\prime }$ if there is a $\eta >0$ such $%
\mathbb{E}\left\vert X_{1}^{\eta }\right\vert <\infty $. Conversely, the
process has no support in $\mathcal{S}^{\prime }$ if $\mathbb{E}\left\vert
X_{1}^{\eta }\right\vert \rightarrow \infty $ for any $\eta >0$. To show
that the second statement in the Dalang-Humeau theorem is not empty amounts
to find a process $X_{t}$ with L\'{e}vy measure $\nu \left( dx\right) $%
\begin{equation}
\int \left( x^{2}\wedge 1\right) \nu \left( dx\right) <\infty  \label{1.2}
\end{equation}%
but for which $\mathbb{E}\left\vert X_{1}^{\eta }\right\vert \rightarrow
\infty $ for any $\eta >0$.

Assuming that such a process can be found, it would be interesting to
characterize its support. A possible candidate space, intermediate between $%
\mathcal{S}^{\prime }$ and $\mathcal{D}^{\prime }$ is the $\mathcal{K}%
^{\prime }$ space of distributions of exponential type, dual to the $%
\mathcal{K}$ space of functions topologized by the norms \cite{Silva1} (see
also \cite{Hoskins}, ch. 3.6)%
\begin{equation}
\left\Vert \varphi \right\Vert _{p}=\max_{0\leq q\leq p}|\sup_{x\in \mathbb{R%
}}\left( e^{p\left\vert x\right\vert }\varphi ^{(q)}\left( x\right) \right) |
\label{1.3}
\end{equation}%
Denoting by $\mathcal{K}_{p}$ the Banach space for the norm $\left\Vert
\cdot \right\Vert _{p}$, $\mathcal{K}^{\prime }$ is the dual of the
countably normed space $\mathcal{K=\cap }_{p=0}^{\infty }\mathcal{K}_{p}$
and is a dense linear subspace of $\mathcal{D}^{\prime }$. Fourier
transforms of distributions in $\mathcal{K}^{\prime }$ are the tempered
ultradistributions in $\mathcal{U}^{\prime }$\ \cite{Silva2}. A distribution 
$\mu \in \mathcal{D}^{\prime }$ is in $\mathcal{K}^{\prime }$ if and only if
it can be represented in the form%
\begin{equation}
\mu
\left( x\right) =D^{r}\left( e^{a\left\vert x\right\vert }f\left( x\right)
\right)  \label{1.4}
\end{equation}%
for some numbers $r\in \mathbb{N}_{0}$ , $a\in \mathbb{R}$ and a bounded
continuous function $f$, with $D$ a derivative in the distributional sense.

\section{The $K_{\protect\alpha }$ processes}

The $K_{\alpha }$ processes are characterized by the triplet $\left( 0,0,\nu
_{\alpha }\right) $ with the following L\'{e}vy measure%
\begin{equation}
\begin{array}{lll}
\nu _{\alpha }\left( dx\right) =\frac{1}{\left( 1+\left\vert x\right\vert
\right) \log ^{1+\alpha }\left( 1+\left\vert x\right\vert \right) }dx &  & 
0<\alpha <2%
\end{array}
\label{2.1}
\end{equation}%
with $\nu _{\alpha }\left( \mathbb{R}\right) \rightarrow \infty $

The case $\alpha =1$ was considered by Fristedt \cite{Fristedt-2} for the
construction of a counter example but, as far as I know, no further studies
have been made on these processes.

\begin{proposition}
The paths of a $K_{\alpha }$ process are a. s. not supported in $\mathcal{S}%
^{\prime }$
\end{proposition}

\begin{proof}
Let 
\begin{equation}
K_{\alpha }\left( t\right) =L_{\alpha }^{M}\left( t\right) +L_{\alpha
}^{P}\left( t\right)   \label{2.2}
\end{equation}%
be the L\'{e}vy-It\^{o} decomposition of the $K_{\alpha }$ process, $%
L_{\alpha }^{M}$ being the compensated square integrable martingale
containing the small jumps of $K_{\alpha }$ and $L_{\alpha }^{P}$ the
compound Poisson process containing the large jumps. Because $L_{\alpha
}^{M}\left( t\right) $ is a. s. in $\mathcal{S}^{\prime }$, it remains to
analyze the $L_{\alpha }^{P}$ component, with L\'{e}vy measure%
\begin{equation}
\nu _{\alpha }^{\prime }\left( dx\right) =\boldsymbol{1}_{\left\vert
x\right\vert >1}\nu _{\alpha }\left( dx\right)   \label{2.3}
\end{equation}%
The L\'{e}vy measure $\nu ^{\prime }\left( dx\right) $ has no positive
moments. Let $\eta >0$%
\begin{eqnarray*}
\int_{\mathbb{R}}\mathbf{1}_{\left\vert x\right\vert \geq 1}\left\vert
x\right\vert ^{\eta }\nu _{\alpha }\left( dx\right)  &=&2\int_{1}^{\infty }%
\frac{x^{\eta }}{\left( 1+x\right) \log ^{1+\alpha }\left( 1+x\right) }dx \\
&>&2\int_{1}^{x^{\ast }}\frac{x^{\eta }}{\left( 1+x\right) \log ^{1+\alpha
}\left( 1+x\right) }dx+2\int_{x^{\ast }}^{\infty }\frac{dx}{\left(
1+x\right) \log \left( 1+x\right) } \\
&=&2\int_{1}^{x^{\ast }}\frac{x^{\eta }}{\left( 1+x\right) \log ^{1+\alpha
}\left( 1+x\right) }dx+2\left. \log \left( \log \left( 1+x\right) \right)
\right\vert _{x^{\ast }}^{\infty }\rightarrow \infty 
\end{eqnarray*}%
$x^{\ast }$ being the solution of%
\begin{equation*}
x^{\ast }=\log ^{\frac{\alpha }{\eta }}\left( 1+x^{\ast }\right) 
\end{equation*}

Furthermore, the rate of growth of the supremum $L^{\ast }\left( t\right)
=\sup_{0\leq s\leq t}\left\vert L_{\alpha }^{P}\left( t\right) \right\vert $
of the modulus process may also be analyzed by computing its index (\cite%
{Pruitt}, \cite{Sato} ch. 48) 
\begin{equation}
\overline{h}\left( r\right) =\int_{\left\vert x\right\vert >r}\nu _{\alpha
}^{\prime }\left( dx\right) +r^{-2}\int_{\left\vert x\right\vert \leq
r}\left\vert x\right\vert ^{2}\nu _{\alpha }^{\prime }\left( dx\right)
+r^{-1}\int_{1<\left\vert x\right\vert \leq r}x\nu _{\alpha }^{\prime
}\left( dx\right)   \label{2.5}
\end{equation}%
Because all terms are positive one has%
\begin{equation}
\overline{h}\left( r\right) >\int_{\left\vert x\right\vert >r}\nu _{\alpha
}^{\prime }\left( dx\right) =\frac{1}{\alpha \log ^{\alpha }\left(
1+r\right) }  \label{2.6}
\end{equation}%
Then the index $\overline{\beta }_{L}$ is%
\begin{equation}
\overline{\beta }_{L}=\sup \left\{ \eta :\lim \sup_{r\rightarrow \infty
}r^{\eta }\overline{h}\left( r\right) =0\right\} =0  \label{2.7}
\end{equation}%
Therefore by proposition 48.10 in \cite{Sato}%
\begin{equation}
\lim \sup_{t\rightarrow \infty }t^{-1/\eta }L^{\ast }\left( t\right) =\infty 
\label{2.8}
\end{equation}%
for any positive $\eta $. That is, the process is not slowly growing.

Dalang and Humeau \cite{Dalang-Humeau} have proved that if a process has
support in $\mathcal{S}^{\prime }$ there is a set of probability one for
which the function $t\rightarrow L_{\alpha }^{P}\left( t\right) $ is slowly
growing. Therefore the paths of $L_{\alpha }^{P}\left( t\right) $ are almost
surely not supported in $\mathcal{S}^{\prime }$.
\end{proof}

\begin{proposition}
The paths of the $K_{\alpha }$ process, for $1<\alpha <2$, have a. s.
support in $\mathcal{K}^{\prime }$
\end{proposition}

\begin{proof}
Here one considers only the positive large jumps or equivalently the modulus
process $\left\vert L_{\alpha }^{P}\left( t\right) \right\vert $. This
process is a subordinator and to find its support one looks for an
appropriate upper function. This subordinator has Laplace exponent%
\begin{equation}
\Phi \left( \lambda \right) =\int_{\left( 0,\infty \right) }\left(
1-e^{-\lambda x}\right) \nu _{\alpha }^{\prime }\left( dx\right)  \label{2.9}
\end{equation}%
and the tail of the L\'{e}vy measure is%
\begin{equation}
\overline{\nu _{\alpha }^{\prime }}\left( x\right) =\nu _{\alpha }^{\prime
}\left( x,\infty \right) =\frac{1}{\alpha \log ^{\alpha }\left( 1+x\right) }
\label{2.10}
\end{equation}%
Now one looks for an upper function and uses Theorem 13 in \cite{Bertoin}.
Let $f\left( x\right) $ be an increasing function that increases faster than 
$x$ and compute%
\begin{equation}
I_{\alpha }\left( f\right) =\int_{1}^{\infty }\overline{\nu _{\alpha
}^{\prime }}\left( f\left( x\right) \right) dx  \label{2.11}
\end{equation}%
For $f\left( x\right) =x^{\beta }$ $\left( \beta >1\right) $ this integral
is $\infty $, hence $\lim \sup_{t\rightarrow \infty }\left( \left\vert
L_{\alpha }^{P}\left( t\right) \right\vert /x^{\beta }\right) =\infty $ a.
s. consistent with no support in $\mathcal{S}^{\prime }$. However if $%
f\left( x\right) =\exp \left( c\left\vert x\right\vert \right) $, $c>0$, $%
I_{\alpha }\left( f\right) <\infty $ for $1<\alpha <2$. Hence in this case,
by Theorem 13 in \cite{Bertoin}, $\lim \sup_{t\rightarrow \infty }\left(
\left\vert L_{\alpha }^{P}\left( t\right) \right\vert /\exp \left(
c\left\vert x\right\vert \right) \right) =0$, $\exp \left( c\left\vert
x\right\vert \right) $ is an upper function for the process and from the
definition (\ref{1.3}) of the norms in $\mathcal{K}$ it follows that the
process has a. s. support in $\mathcal{K}^{\prime }$.
\end{proof}

For the $0<\alpha \leq 1$ case it is also possible to pinpoint a precise
support for the $K_{\alpha }$ processes. For each $\alpha $ one defines a
family of norms%
\begin{equation}
\left\Vert \varphi \right\Vert _{p,\beta }=\max_{0\leq q\leq p}|\sup_{x\in 
\mathbb{R}}\left( e^{p\left\vert x\right\vert ^{\beta }}\varphi ^{(q)}\left(
x\right) \right) |,\;\;\;\beta >1  \label{2.12}
\end{equation}%
with $\beta >\frac{1}{\alpha }$, denotes by $\mathcal{K}_{p,\beta }$ the
Banach space associated to the norm $\left\Vert \cdot \right\Vert _{p,\beta
} $ and by $\mathcal{K}_{\beta }$ the countably normed space $\mathcal{K}%
_{\beta }\mathcal{=\cap }_{p=0}^{\infty }\mathcal{K}_{p,\beta }$. The dual
of $\mathcal{K}_{\beta }$, which one denotes as $\mathcal{K}_{\beta
}^{\prime }$, provides a distributional support for the $K_{\alpha }$
processes whenever $\alpha >\frac{1}{\beta }$.

For many applications both the properties of L\'{e}vy processes and L\'{e}vy
white noise are needed. For each event $\omega $ in the probability space, $%
K_{\alpha }-$white noise is defined by the derivative in the sense of
distributions%
\begin{equation}
\left\langle \overset{\cdot }{K}_{\alpha }\left( \omega \right) ,\varphi
\right\rangle :=-\left\langle K_{\alpha }\left( \omega \right) ,\varphi
^{\prime }\right\rangle :=\int_{\mathbb{R}_{+}}K_{\alpha }\left( t,\omega
\right) \varphi ^{\prime }\left( t\right) dt  \label{2.13}
\end{equation}%
$\varphi $ being a function in $\mathcal{K}$ or $\mathcal{K}_{\beta }$.
Because there is no upper bound on the order of the derivatives in (\ref{1.3}%
) or (\ref{2.12}) one concludes that $K_{\alpha }-$white noise has the same
support properties as the $K_{\alpha }$ processes.

\end{document}